\documentclass[11pts]{amsart}

\usepackage{amsmath,amsthm,amscd,amssymb,bbm,color}
\usepackage{latexsym}
\usepackage{fullpage}
\usepackage{tikz-cd}
\usepackage{mathtools}
\usepackage[toc]{appendix}

\usepackage{setspace}
\usepackage{nohyperref}

\usepackage{systeme}

\usepackage [english]{babel}
\usepackage [autostyle, english = american]{csquotes}
\MakeOuterQuote{"}

\numberwithin{equation}{section}

\newtheorem{lem}{Lemma}[section]
\newtheorem{axiom}{Axiom}[section]

\newtheorem{thm}{Theorem}[section]
\newtheorem{defn}{Definition}[section]
\newtheorem{cor}{Corollary}[section]
\newtheorem{conj}{Conjecture}[section]

\newtheorem{rem}{Remark}[section]

\usepackage{thmtools,thm-restate}

\DeclareMathOperator{\C}{\mathbb{C}}

\DeclareMathOperator{\Z}{\mathbb{Z}}
\DeclareMathOperator{\N}{\mathbb{N}}

\DeclareMathOperator{\lcm}{lcm}

\newcommand{\Mod}[1]{\ (\text{mod}\ #1)}

\newcommand{\tends}{\rightarrow}

\title{Siegel zeros and Sarnak's Conjecture}
\author{Jake Chinis\vspace{-1.5\baselineskip}}

\begin{document}

\maketitle

\begin{abstract}
    Assuming the existence of Siegel zeros, we prove that there exists an increasing sequence of positive integers for which Chowla's Conjecture on $k$-point correlations of the Liouville function holds. This extends work of Germ{\'a}n and K{\'a}tai, where they studied the case $k=2$ under identical hypotheses. 
    
    An immediate corollary, which follows from a well-known argument due to Sarnak, is that Sarnak's Conjecture on M{\"o}bius disjointness holds. More precisely, assuming the existence of Siegel zeros, there exists a subsequence of the natural numbers for which the Liouville function is asymptotically orthogonal to any sequence of topological entropy zero. 
\end{abstract}

\section{Introduction}

The Liouville function, $\lambda$, is the completely multiplicative function defined by $\lambda(p):=-1$ for all primes $p$. As such, $\lambda(n)=1$ if $n$ has an even number of prime factors, counted with multiplicity, and $\lambda(n)=-1$ otherwise. An important problem in analytic number theory is to understand the asymptotic behaviour of the partial sums $\sum_{n\leq x}\lambda(n)$. Viewing $\{\lambda(n)\}_n$ as a sequence of independent random variables, each taking the value $\pm 1$ with equal probability, we might expect that the partial sums of $\lambda$ exhibit some cancellation. Indeed, the Prime Number Theorem (PNT) is equivalent to the fact  
\begin{align*}
  \sum_{n\leq x} \lambda(n)=o(x),
\end{align*}
as $x\tends \infty$. The Riemann Hypothesis (RH) is equivalent to the fact that the partial sums of $\lambda$ exhibit "square-root cancellation:" for any $\epsilon >0$,
\begin{align*}
       \sum_{n\leq x}\lambda(n) \ll x^{1/2+\epsilon},
\end{align*}
as $x\tends \infty$. A famous conjecture due to Chowla \cite{ChowlaConj} states similar estimates hold for \textit{multiple correlations} of $\lambda$:

\begin{conj}[Chowla's Conjecture]
For any $k$ distinct integers $h_1,\dots,h_k$,
\begin{align*}
    \sum_{n\leq x} \lambda(n+h_1)\cdots \lambda(n+h_k)=o_k(x),
\end{align*}
as $x\tends \infty$, where we set $\lambda(n)=0$ for integers $n\leq 0$ and where the implied constant may depend on $h_1,\dots,h_k$.
\end{conj}

As stated earlier, the case $k=1$ is equivalent to the PNT. Chowla's Conjecture remains open for $k>1$ and the case $k=2$ is believed to be as difficult as the Twin Prime Conjecture; see \cite{Chowla+TwinPrimes}.

\begin{rem}
There are analogues of Chowla's Conjecture over finite function fields. We refer the reader to the work of Sawin and Shusterman \cite{Chowla_Finite}, where they study the relationship between Chowla's Conjecture and the Twin Prime Conjecture over finite function fields. It is interesting to note that Sawin and Shusterman use the same general idea that we use in this paper (and that was also used in \cite{Liou_Siegel}); namely, they approximate the M{\"o}bius function by a Dirichlet character.
\end{rem}

Although Chowla's Conjecture still seems to be out of reach, there has been much progress towards partial results of this conjecture, as well as proofs of variations of Chowla's original claim. For instance, Harman, Pintz, and Wolke \cite{HPW_Liou} proved that 
\begin{align*}
    -\frac{1}{3}+\mathcal{O}\Big(\frac{\log x}{x}\Big)
    \leq
    \frac{1}{x}\sum_{n\leq x}\lambda(n)\lambda(n+1)
    \leq
    1-\mathcal{O}_\epsilon\Big(\frac{1}{\log^{7+\epsilon}x}\Big),
\end{align*}
for all $\epsilon>0$. This was subsequently improved by Matom{\"a}ki and Radziwi{\l\l} \cite{MatRad_Mult} to
\begin{align*}
    \frac{1}{x}\Biggr|\sum_{n\leq x}\lambda(n)\lambda(n+1)\Biggr|\leq 1-\delta,
\end{align*}
for some explicit constant $\delta>0$ and all $x$ sufficiently large.

Concerning weaker versions of Chowla's Conjecture, Matom{\"a}ki, Radziwi{\l\l}, and Tao \cite{MatRadTao_AvgChowla} averaged over the parameters $h_1,\dots, h_k$ and showed that, for any $k\in\N$ and any $10\leq h \leq x,$
\begin{align*}
    \sum_{1\leq h_1,\dots,h_k\leq h}\Biggr|\sum_{n\leq x}\lambda(n+h_1)\cdots \lambda(n+h_k)\Bigg|
    \ll
    k\Biggr(\frac{\log \log h}{\log h}+\frac{1}{\log^{1/3000}x}\Biggr)h^kx,
\end{align*}
thus establishing an \textit{averaged} form of Chowla's Conjecture.

Tao \cite{Log_Chowla_1} made progress towards a \textit{logarithmically averaged} version of Chowla's Conjecture by showing that
\begin{align*}
    \sum_{n\leq x}\frac{\lambda(n)\lambda(n+1)}{n}=o(\log x)
\end{align*}
as $x\tends \infty.$ Following up on this, Tao and Ter{\"a}v{\"a}inen \cite{Log_Chowla_3,Log_Chowla_2} were able to establish a logarithmically averaged version of Chowla's Conjecture for odd $k$-point correlations; that is, for any odd $k\in\N$ and any integers $h_1,\dots,h_k$,
\begin{align*}
    \sum_{n\leq x}\frac{\lambda(n+h_1)\cdots \lambda(n+h_k)}{n}=o(\log x)
\end{align*}
as $x\tends \infty$. Recently, Helfgott and Radziwi{\l}{\l} \cite{Helfgott-Maks_LogChowla} improved the bounds obtained by Tao in \cite{Log_Chowla_1} and Tao and Ter{\"a}v{\"a}inen in \cite{Log_Chowla_2} for $k=2$.

\begin{rem}
Note that the preceding results follow from Chowla's Conjecture either immediately or by partial summation. Furthermore, the same results can be stated for the M{\"o}bius function, $\mu$, in place of $\lambda$, via the identity $\mu(n) = \sum_{d^2|n}\mu(d)\lambda\left(\frac{n}{d^2}\right)$.
\end{rem}

In this paper, we are concerned with the relationship between the Liouville function and Siegel zeros. Our ultimate aim is to extend the work of Germ{\'a}n and K{\'a}tai \cite{Liou_Siegel}, where they studied $2$-point correlations of the Liouville assuming the existence of Siegel zeros:

\begin{thm}\cite[Theorem 2]{Liou_Siegel}
\label{Liou_Siegel}
Let $\{q_\ell\}_\ell$ be an increasing sequence of positive integers with corresponding sequence of real primitive characters $\{\chi_\ell \Mod{q_\ell}\}_\ell$. Suppose that $L(s,\chi_\ell)$ has a Siegel zero $\beta_\ell:=1-\frac{1}{\eta_\ell\log q_\ell}$ with $\eta_\ell>\exp\exp (30)$ for all $\ell\in\N$. Then, there exists a constant $c>0$ and a function $\varepsilon(x)\tends 0$ as $x\tends \infty$ such that
\begin{align*}
    \frac{1}{x}\Biggr|\sum_{n\leq x}\lambda(n)\lambda(n+1)\Biggr|\leq \frac{c}{\log\log \eta_\ell} + \varepsilon(x),
\end{align*}
uniformly for $x\in[q_\ell^{10},q_\ell^{(\log\log \eta_\ell)/3}]$.
\end{thm}

The key to the work of Germ{\'a}n and K{\'a}tai is to approximate $\lambda$ by $\chi_\ell$ on "large" primes via the completely multiplicative function $\lambda_r$ defined by
\begin{align*}
    \lambda_r(p)
    :=
    \left\{\def\arraystretch{1.2}%
    \begin{array}{@{}l@{\quad}l@{}}
        \lambda(p) & \mbox{if $p\leq r,$}\\
        \chi_\ell(p) & \mbox{if $p>r$,}
    \end{array}\right.
\end{align*}
for some well-chosen parameter $r=r_\ell$. Then, using similar ideas as Heath-Brown in his work on Siegel zeros and the Twin Prime Conjecture \cite{HB_TP}, Germ{\'a}n and K{\'a}tai show that the $2$-point correlations of $\lambda$ are well approximated by the $2$-point correlations of $\lambda_r$, along a subsequence. The added benefit to this approach is that we can now use sieve theory, together with the definition of $\lambda_r$, to relate the $2$-point correlations of $\lambda$ to some character sum, which is known to be small. Following this same line of reasoning, we prove the corresponding result for (general) $k$-point correlations:

\begin{thm}
\label{main_thm}
Let $\{q_\ell\}_\ell$ be an increasing sequence of positive integers with corresponding sequence of real primitive characters $\{\chi_\ell \Mod{q_\ell}\}_\ell$. Suppose that $L(s,\chi_\ell)$ has a Siegel zero $\beta_\ell:=1-\frac{1}{\eta_\ell\log q_\ell}$ with $\eta_\ell>\exp\exp (30)$ for all $\ell\in\N$. Then, for any distinct (positive) integers $h_1,\dots,h_k$, there exists a constant $c_k=c(h_1,\dots,h_k)>0$ such that
\begin{align*}
    \frac{1}{x}\Biggr|\sum_{n\leq x}\lambda(n+h_1)\cdots\lambda(n+h_k)\Biggr|
    \leq
    \frac{c_k}{(\log\log\eta_\ell)^{1/2}(\log\eta_\ell)^{1/12}},
\end{align*}
uniformly for $x\in[q_\ell^{10},q_\ell^{(\log\log \eta_\ell)/3}]$.
\end{thm}

\begin{rem}
Since $\eta_\ell\rightarrow\infty$ as $\ell\rightarrow\infty$ (see Section \ref{Siegel}), Theorem \ref{main_thm} thus establishes Chowla's Conjecture along a subsequence, assuming the existence of Siegel zeros. Furthermore, one should think of Theorem \ref{Liou_Siegel} as the multiplicative analogue of Heath-Brown's result on Twin Primes and Siegel zeros \cite{HB_TP}, while Theorem \ref{main_thm} is the multiplicative analogue of the Hardy\textendash Littlewood $k$-tuples conjecture (which is also known to hold, assuming the existence of Siegel zeros).
\end{rem}

\begin{rem}
Notice that Theorem \ref{main_thm} is an improvement on Theorem \ref{Liou_Siegel} in two respects: first, we can handle general $k$-point correlations (as opposed to the case where $k=2$); further, we have an exponential improvement in our bounds (which follows from using a different version of the Fundamental Lemma of Sieve Theory and from taking a different choice of $r$ than those used in \cite{Liou_Siegel}; see Appendix \ref{Sieve_Theory} and the end of Section \ref{Proof}). 

Note further that the work in \cite{Liou_Siegel} deals only with $h_1=0$ and $h_1=1$, but the proof extends easily to general $h_1$, $h_2$ (with minor modifications). The main difficulty in going from the case $k=2$ to general $k$-point correlations lies in being able to parametrize integer solutions of the following system of linear equations, in the unknowns $x_0,x_1,\dots x_k$, for some integers $a_0,a_1,\dots, a_k$:
\begin{align*}
        \left\{\def\arraystretch{1.2}%
        \begin{array}{@{}l@{\quad}l@{}}
            a_1x_1 = a_0x_0 + h_1,\\
            \vdots\\
            a_kx_k = a_0x_0 + h_k.
        \end{array}\right.
\end{align*}
It is to verify that, if this system is solvable, then the solutions are given by $x_i=x_i^*+m\lcm(a_0,a_1,\dots,a_k)/a_i$, where $(x_0^*,x_1^*,\dots,x_k^*)$ is one particular solution and $m\in \Z$ (essentially generalizing Bezout's Identity to $k$ equations). From there, we need to bound character sums evaluated at the polynomial $f(n):=(x_0^*+na_0^*)\cdots (x_k^*+na_k^*),$ where $a_i^*:=\lcm(a_0,\dots,a_k)/a_i$, as $n$ varies over one complete residue class modulo $q_\ell$. Fortunately for us, these character sums exhibit squareroot cancellation via Weil's Bound, provided that $f$ is not a square. For more details, see Appendices \ref{Character_Sums} and \ref{Parametrization}.
\end{rem}

%%%%%

\subsection{Sarnak's Conjecture}

We should think of the previous results as instances of the so-called "M{\"o}bius Randomness Law," which states that the values of $\lambda$ (or $\mu$) are random enough so that the twisted sums $\sum_{n\leq x}\lambda(n)a_n$ should be small for any "reasonable" sequence of complex numbers $\{a_n\}_n$; see \cite[Section 13]{IwaKow}. A famous conjecture due to Sarnak characterizes one such family of "reasonable" sequences as those which are \textit{deterministic}:

\begin{defn}
Given a bounded sequence $f:\N\rightarrow \C$, its \textbf{topological entropy} is equal to the least exponent $\sigma$ for which the set
\[
\{(f(n+1),f(n+1),\dots,f(n+m))\}_{n=1}^\infty\subset \C^m
\]
can be covered by $\mathcal{O}(\exp(\sigma m + o(m))$ balls of radius $\epsilon$ (in the $\ell^\infty$ metric), for any fixed $\epsilon>0$, as $m\rightarrow\infty$. In the case where $\sigma=0$, we say that $f$ is \textbf{deterministic}.
\end{defn}

\begin{conj}[Sarnak's Conjecture]
Let $f:\N\rightarrow \C$ be a deterministic sequence. Then,
\begin{align*}
    \sum_{n\leq x}\lambda(n)f(n) = o_f(x),
\end{align*}
as $x\rightarrow\infty.$
\end{conj}

Although Sarnak's Conjecture has yet to be resolved, there are many instances for which the conjecture holds. For example, in the case where $f$ is constant, Sarnak's Conjecture is equivalent to the PNT; in the case where $f$ is periodic, it is equivalent to the PNT in arithmetic progressions. For a more thorough survey on various instances for which Sarnak's Conjecture holds, see \cite{Survey-Sarnak-1,Survey-Sarnak-2}. 

By a well-known argument due to Sarnak, we also know that Chowla's Conjecture implies Sarnak's Conjecture; as a result, Theorem \ref{main_thm} yields the following:

\begin{cor}
\label{cor_Sarnak}
Let $f:\N\rightarrow \C$ be a deterministic sequence. Under the hypotheses of Theorem \ref{main_thm},
\begin{align*}
    \sum_{n\leq x}\lambda(n)f(n) = o_f(x),
\end{align*}
for $x\in[q_\ell^{10},q_\ell^{(\log\log \eta_\ell)/3}]$.
\end{cor}
\begin{proof}
The proof follows Sarnak's argument verbatim, the details of which can be found on Tao's blog\footnote{https://terrytao.wordpress.com/2012/10/14/the-chowla-conjecture-and-the-sarnak-conjecture/}. For further work on the relationship between Chowla's Conjecture and Sarnak's Conjecture, see \cite{Chowla_Exposition,Chowla-Sarnak_Subsequence,Sarnak-Subsequence}.
\end{proof}

%%%%

\subsection{Outline} 
Our paper is split as follows: in Section \ref{Siegel}, we give a brief introduction on Siegel zeros and the work of Heath-Brown on counting the number of primes $p$ such that $\chi(p)=1$; in Section \ref{Reduction}, we use the work of Germ{\'a}n\textendash K{\'a}tai/Heath-Brown to relate the $k$-point correlations of $\lambda$ to those of $\lambda_r$; from there, we use some estimates on character sums, sieve theory, and elementary number theory to prove Theorem \ref{main_thm}; Appendices \ref{Sieve_Theory}, \ref{Character_Sums}, and \ref{Parametrization} contain the relevant background information on sieve theory, character sums, and Diophantine equations, which we use freely in the proof of Theorem \ref{main_thm}.
\section{Siegel zeros and primes $p$ such that $\chi(p)=1$}
\label{Siegel}

In this section, we provide a brief introduction to Siegel zeros, culminating in the work of Heath-Brown on primes $p$ such that $\chi(p)=1$. To begin, we must first discuss zero-free regions of Dirichlet $L$-functions associated to Dirichlet characters $\chi \Mod{q}$; we follow Chapter 12 of \cite{Dimitris}:

\begin{thm}
\label{Exc-Thm}
Let $q\geq 3$ and set $Z_q(s):=\prod_{\chi \Mod{q}}L(s,\chi)$. Then, there is an absolute constant $c>0$ such that the region $\Re(s) \geq 1-\frac{c}{\log(q\tau)},$ where $\tau=\max\{1,|\Im(s)|\}$, contains at most one zero of $Z_q$. Furthermore, if this exceptional zero exists, then it is necessarily a real, simple zero of $Z_q$, say $\beta_1\in[1-c/\log q,1]$, and there is a real, non-principal character $\chi_1\Mod{q}$ such that $L(\beta_1,\chi_1)=0$.
\end{thm}
\begin{proof}
See \cite[Theorem 12.3]{Dimitris}, for example.
\end{proof}

We call the character $\chi_1$ in Theorem \ref{Exc-Thm} an \textit{exceptional character} and its zero, $\beta_1$, is the associated \textit{exceptional zero}, or \textit{Siegel/Landau\textendash Siegel zero}. Note that this exceptional character depends on the choice of absolute constant and that this relationship implies some interesting facts:
\begin{enumerate}
\item If we have one exceptional character, then we actually have infinitely many exceptional characters: if we had only finitely many exceptional characters $\chi_i\Mod{q_i}$, we could set $c^\prime:=\frac{1}{2}\min_i\{(1-\beta_i)\log q_i\}$ and we would then have that
\[
1-\frac{c}{\log q_i\tau} \leq \beta_i < 1-\frac{c^\prime}{\log q_i\tau}
\]
for all $i$; in particular, replacing $c$ with $c^\prime$ in Theorem \ref{Exc-Thm}, we no longer have any exceptional zeros.
\item Similarly, we can take $c$ to be arbitrarily small: if there are no exceptional zeros for $c$ small enough, then we are done.
\end{enumerate}

Thus, when we talk about Siegel zeros/exceptional characters, we are actually talking about an infinite sequence of real, primitive Dirichlet characters $\{\chi_\ell\Mod{q_\ell}\}_{\ell=1}^\infty$ for which $L(s,\chi_\ell)$ has a real zero 
\begin{align}
\label{beta_l}
\beta_\ell=1-o_{\ell\rightarrow\infty}\left(\frac{1}{\log q_\ell}\right),
\end{align}
and such that no product $\chi_\ell\chi_{\ell^\prime}$ is principal for any $\ell\neq \ell^\prime$. Using Siegel's Theorem, we can quantify the rate of convergence in Equation (\ref{beta_l}):

\begin{thm}[Siegel]
Let $\epsilon > 0$. Then, there is a constant $c(\epsilon) > 0$, which
cannot be computed effectively, such that $L(\sigma,\chi)\neq 0$ for $\sigma > 1-c(\epsilon)q^{-\epsilon}$ and for all real, non-principal Dirichlet characters $\chi \Mod{q}$.
\end{thm}
\begin{proof}
See \cite[Theorem 12.10]{Dimitris}, for example.
\end{proof}

In particular,
\begin{align}
    \eta_\ell
    :=
    ((\beta_\ell -1)\log q_\ell)^{-1}
    \ll q_\ell,
\end{align}
as $\ell\rightarrow \infty$. In fact, one could show that $\eta_\ell\ll$ any fixed power of $q_\ell$, but the above is all we need for our purposes.

Now that we know exactly what we mean by Siegel zeros/exceptional characters, we can consider consequences of their existence. For example, Heath-Brown \cite{HB_TP} showed, under similar hypotheses to Theorem \ref{main_thm}, that the existence of Siegel zeros implies the Twin Prime Conjecture. Recently, Granville \cite{Granville_Siegel/Sieve} used the existence of Siegel zeros to study problems in sieve theory, such as improving (conditionally) lower bounds on the longest gaps between primes. For our purposes, we are interested in the following lemma, due to Germ{\'a}n and K{\'a}tai, which is a variation of Lemma 3 in \cite{HB_TP}: 

\begin{lem}[\cite{Liou_Siegel}]
\label{GK-Primes}
Let $\{\chi_\ell\Mod{q_\ell}\}_\ell$ denote a sequence of exceptional characters with corresponding Siegel zero 
\[
\beta_\ell=1-\frac{1}{\eta_\ell\log q_\ell},
\]
with $\eta_\ell>\exp(\exp(30))$. Then,
\[
\sum_{\substack{{p\leq x}\\{\chi_\ell(p)=1}}}\frac{\log p}{p}
\ll
\exp\left(\frac{\log x}{\log q_\ell}\right)(\log q_\ell)(\log\eta_\ell)^{-1/2},
\]
uniformly for $x\in [q_\ell^{10},q_\ell^{(\log\log\eta_\ell)/3}].$
\end{lem}

\begin{rem}
Note that the upper bound in Lemma \ref{GK-Primes} is worse than that in Lemma 3 of \cite{HB_TP}, but the range of admissible $x$ is larger: Lemma 3 of \cite{HB_TP} yields the upper bound $\ll (\log q_\ell)(\log\eta_\ell)^{-1/2}$, uniformly for $x\in [q_\ell^{250}, q_\ell^{500}]$ (so that Lemma \ref{GK-Primes} recovers Heath-Brown's result when $x$ is restricted to the interval $[q_\ell^{250},q_\ell^{500}]$).
\end{rem}

With Lemma \ref{GK-Primes} in tow, we can now approximate $\lambda$ by a Dirichlet character on large primes; this is done in the next section.
\section{Going from $\lambda$ to $\lambda_r$}
\label{Reduction}

From now on, we fix a character $\chi_\ell\Mod{q_\ell}$, so that we may drop the dependence on $\ell$. Using Lemma \ref{GK-Primes}, we can relate the $k$-point correlations of the Liouville function to the $k$-point correlations of the completely multiplicative function $\lambda_r$, which is defined by
\begin{align}
\label{lambda_r}
    \lambda_r(p):=
        \left\{\def\arraystretch{1.2}%
        \begin{array}{@{}l@{\quad}l@{}}
            \lambda(p) & \mbox{if $p\leq r$}\\
            \chi(p) & \mbox{if $p>r$},
        \end{array}\right.
\end{align}
where $r:=x^{1/((\log\log\eta)^{1/2}(\log \eta)^{1/12})}$. The details can be found in pages 48-50 of \cite{Liou_Siegel}; we reproduce their results here, for convenience/completeness\footnote{In \cite{Liou_Siegel}, the authors take $:r=x^{1/\log\log\eta}$, which produces an error of size $1/\log\log\eta$ in Theorem \ref{Liou_Siegel}. Making a different choice of $r$ and using a different version of the FLST allows us to obtain better estimates; see the very end of Section \ref{Proof} for why the choice of $r=x^{1/((\log\log\eta)^{1/2}(\log\eta)^{1/12})}$ was made/is optimal.}:

\begin{lem}
\label{main_lemma}
Suppose $h_1,\dots, h_k$ are distinct (positive) integers. Set $\lambda(n;k):=\lambda(n+h_1)\cdots\lambda(n+h_k)$ and define $\lambda_r(n;k)$ in the same way. Then, under the assumptions of Theorem \ref{main_thm},
\[
\frac{1}{x}\sum_{n\leq x}\lambda(n;k)
=
\frac{1}{x}\sum_{n\leq x}\lambda_r(n;k)
+
\mathcal{O}_k\left(\frac{1}{(\log\log\eta)^{1/2}(\log \eta)^{1/12}}\right),
\]
uniformly for $x\in[q^{10},q^{(\log\log\eta)/{3}}]$.
\end{lem}
\begin{proof}
 Note that
\begin{align*}
    \sum_{n\leq x}\lambda(n;k)
    &=
    \sum_{n\leq x}\Big(\lambda(n;k)
    \pm
    \lambda_r(n+h_1)\lambda(n+h_2)\cdots\lambda(n+h_k)\Big)\\
    &=
    \sum_{n\leq x}\Big(\lambda(n+h_1)-\lambda_r(n+h_1)\Big)\lambda(n+h_2)\cdots\lambda(n+h_k)
    +
    \sum_{n\leq x}\lambda_r(n+h_1)\lambda(n+h_2)\cdots\lambda(n+h_k).
\end{align*}
Continuing by induction, we have that
\begin{align*}
    \left|\sum_{n\leq x}\Big(\lambda(n;k)-\lambda_r(n;k)\Big)\right|
    &\leq
    \sum_{i=1}^k\sum_{n\leq x}|\lambda(n+h_i)-\lambda_r(n+h_i)|\\
    &=
    k\sum_{n\leq x}|\lambda(n)-\lambda_r(n)| + \mathcal{O}_k(1),
\end{align*}
where the last line follows from the "approximate translation-invariance" of the partial sums, noting that the error term depends on $h_1,\dots, h_k$. To bound $\sum_{n\leq x}|\lambda(n;k)-\lambda_r(n;k)|,$ recall the definition of $\lambda_r$ from Equation (\ref{lambda_r}) and note that
\begin{align*}
    \sum_{n\leq x}|\lambda(n)-\lambda_r(n)|
    &=
    \sum_{n\leq x}\Biggr|
        \prod_{p^\alpha||n}\lambda(p^\alpha)
        -
        \prod_{\substack{{p^\alpha||n}\\{p\leq r}}}\lambda(p^\alpha)\prod_{\substack{{p^\alpha||n}\\{p> r}}}\chi(p^\alpha)\Biggr|\\
    &=
    \sum_{n\leq x}\Biggr|
        \prod_{\substack{{p^\alpha||n}\\{p\leq r}}}\lambda(p^\alpha)
        \Biggr(
        \prod_{\substack{{p^\alpha||n}\\{p> r}}}\lambda(p^\alpha)
        -
        \prod_{\substack{{p^\alpha||n}\\{p> r}}}\chi(p^\alpha)
        \Biggr)
        \Biggr|\\
    &=    
    \sum_{n\leq x}\Biggr|
        \prod_{\substack{{p^\alpha||n}\\{p> r}}}\lambda(p^\alpha)
        -
        \prod_{\substack{{p^\alpha||n}\\{p> r}}}\chi(p^\alpha)
        \Biggr|.\\    
\end{align*}
Then, using the fact that
\[
    \left|\prod_i x_i - \prod_i y_i\right|
    \leq 
    \sum_i |x_i-y_i|,
\]
for all $x_i,y_i\in\{-1,0,1\}$, we have that
\begin{align*}
    \sum_{n\leq x}\Biggr|
        \prod_{\substack{{p^\alpha||n}\\{p> r}}}\lambda(p^\alpha)
        -
        \prod_{\substack{{p^\alpha||n}\\{p> r}}}\chi(p^\alpha)
        \Biggr|
    &\leq
    \sum_{n\leq x}\sum_{\substack{{p^\alpha||n}\\{p> r}}}
    |\lambda(p^\alpha)-\chi(p^\alpha)|\\
    &\leq 
    \sum_{\substack{{p^\alpha\leq x}\\{p> r}}}|\lambda(p^\alpha)-\chi(p^\alpha)|
    \sum_{\substack{{n\leq x}\\{p^\alpha||n}}}1\\
    &\ll
    x\sum_{\substack{{p^\alpha\leq x}\\{p> r}}}\frac{|\lambda(p^\alpha)-\chi(p^\alpha)|}{p^\alpha}\\
    &=
    x\sum_{r<p\leq x}\frac{|\lambda(p)-\chi(p)|}{p}
    +
    x\sum_{\substack{{p^\alpha\leq x}\\{p> r}\\{\alpha\geq 2}}}\frac{|\lambda(p^\alpha)-\chi(p^\alpha)|}{p^\alpha}.
\end{align*}
The sum over the higher prime powers can be bounded trivially:
\begin{align*}
    x\sum_{\substack{{p^\alpha\leq x}\\{p> r}\\{\alpha\geq 2}}}\frac{|\lambda(p^\alpha)-\chi(p^\alpha)|}{p^\alpha}
    \ll
    x\sum_{p>r}\sum_{\alpha\geq 2}\frac{1}{p^\alpha}
    \ll
    \frac{x}{r}.
\end{align*}
For the sum over the primes, recall that $\lambda(p)=-1$, which yields:
\begin{align*}
    x\sum_{r<p\leq x}\frac{|\lambda(p)-\chi(p)|}{p}
    &=
    x\sum_{\substack{{r<p\leq x}\\{\chi(p)=1}}}\frac{2}{p}
    +
    x\sum_{\substack{{r<p\leq x}\\{\chi(p)=0}}}\frac{1}{p}\\
    &\ll
    \frac{x}{\log r}\sum_{\substack{{r<p\leq x}\\{\chi(p)=1}}}\frac{\log p}{p}
    +
    \frac{x}{\log r}\sum_{\substack{{r<p\leq x}\\{\chi(p)=0}}}\frac{\log p}{p},
\end{align*}
where the last line follows from the fact that $\log p/\log r>1$ for all $p>r$. Since $\chi(p)=0$ iff $p|q$, the sum over primes $p$ such that $\chi(p)=0$ can be bounded above by
\begin{align*}
    \frac{x}{\log r}\sum_{\substack{{r<p\leq x}\\{\chi(p)=0}}}\frac{\log p}{p}
    &\ll
    \frac{x}{\log r}\left(\sum_{p\leq \log q}\frac{\log p}{p}
    +
    \sum_{\substack{{\log q< p \leq x}\\{p|q}}}\frac{\log p}{\log q}\right)\\
    &\ll
    \frac{x}{\log r}\left(\log\log q + 1\right)\\
    &\ll
    \frac{x\log\log q}{\log r}.
\end{align*}
For the remaining sum, we can apply Lemma \ref{GK-Primes}, which yields:
\begin{align*}
    \frac{x}{\log r}\sum_{\substack{{r<p\leq x}\\{\chi(p)=1}}}\frac{\log p}{p}
    &\ll
    \frac{x}{\log r}\exp\left(\frac{\log x}{\log q}\right)(\log q)(\log\eta)^{-1/2}.
\end{align*}
Recalling that $x\in[q^{10},q^{{(\log\log\eta)}/{3}}],$ $r=x^{1/((\log\log\eta)^{1/2}(\log \eta)^{1/12})}$, and $\eta\ll q$, the total error is then bounded above by
\begin{align*}
    \frac{x}{\log r}&\left(\exp\left(\frac{\log x}{\log q}\right)(\log q)(\log\eta)^{-1/2}
    +
    \log\log q\right)
    +
    \frac{x}{r}
    \ll
    \frac{x}{(\log\log\eta)^{1/2}(\log \eta)^{1/12}},
\end{align*}
which follows from the fact that 
\[
\frac{1}{\log r}\exp\left(\frac{\log x}{\log q}\right)(\log q)(\log\eta)^{-1/2}
=
\frac{(\log\log\eta)^{1/2}(\log \eta)^{1/12}}{\log x}\exp\left(\frac{\log x}{\log q}\right)(\log q)(\log\eta)^{-1/2}
\]
is an increasing function of $x$ (for $x\geq q$), whose maximum on the interval $[q^{10},q^{(\log\log\eta)/3}]$ will occur at $x=q^{(\log\log\eta)/3}.$ In any case, we then have that
\begin{align*}
    \frac{1}{x}\sum_{n\leq x}\lambda(n;k)
    =
    \frac{1}{x}\sum_{n\leq x}\lambda_r(n;k)
    +
    \mathcal{O}_k\left(\frac{1}{(\log\log\eta)^{1/2}(\log\eta)^{1/12}}\right),
\end{align*}
as claimed.
\end{proof}

From Lemma \ref{main_lemma}, it now suffices to bound the $k$-point correlations of $\lambda_r$ in order to prove Theorem \ref{main_thm}; the next section is dedicated to this task.

\section{Proof of Theorem \ref{main_thm}}
\label{Proof}

In this section, we complete the proof of Theorem \ref{main_thm} by bounding the $k$-point correlations of $\lambda_r$. Our main tool  is the Fundamental Lemma of Sieve Theory (Lemma \ref{FLST}), but before we can use this, we must first control the so-called "level of distribution" of the sieve; this is done with some preliminary sieving, by removing integers with "small" prime factors:

\begin{lem}[Controlling the level of distribution]
\label{Pre-Sieve}
Let $r<x$ and suppose $A_x\rightarrow \infty$  as $x\rightarrow \infty$. Then\footnote{In \cite{Liou_Siegel}, the bound $x/A_x$ is simply written as $o(x)$, which is where this $\varepsilon(x)$ function comes from in Theorem \ref{Liou_Siegel}. Keeping track of this error and then optimising the choice of $r$ is how we obtain the improvements in Theorem \ref{main_thm}.}:
\[
\#\{n\leq x:\prod_{\substack{{p^\alpha || n}\\{p\leq r}}}p^\alpha>r^{A_x}\}\ll \frac{x}{A_x}.
\]
\end{lem}
\begin{proof}
This follows from Chebyshev's Inequality.
\end{proof}

Using Lemma \ref{Pre-Sieve}, we can now restrict ourselves to integers $n\leq x$ such that the $r$-smooth parts of $n+h_i$ are all bounded above by $r^{A_x}$, for some $A_x$ going to infinity slowly enough with respect to both $x$ and $\eta$:
\[
\sum_{n\leq x}\lambda_r(n;k)
=
\sum_{\substack{{n\leq x}\\{\prod_{\substack{{p^\alpha||(n+h_i)}\\{p\leq r}}}p^\alpha}\leq r^{A_x}}}\lambda_r(n;k)+\mathcal{O}\left(\frac{x}{A_x}\right).
\]
Note: we will eventually choose $A_x\asymp_k (\log\log\eta)^{1/2}(\log\eta)^{1/12}$, which produces an admissible error in Theorem \ref{main_thm}. For simplicity, we assume that $h_1=0$ and relabel the remaining indices: this amounts to shifting the sum over $n$ by $h_1$ (which incurs an error of $\mathcal{O}_k(1)$), so that $h_i := h_i-h_1$ for $i=2,3,\dots,k$. Relabeling the indices as $i=1,2,\dots,k$, it suffices to bound 
\[
\sum_{\substack{{n\leq x}\\{\prod_{\substack{{p^\alpha||(n+h_i)}\\{p\leq r}}}p^\alpha}\leq r^{A_x}}}\lambda(n)\lambda_r(n+h_1)\cdots\lambda_r(n+h_k),
\]
where we should think of $k$ as $k-1$, with $h_0=0$.

Next, write $n$ as $a_0b_0$, where $a_0$ is the $r$-smooth part of $n$ and where $b_0$ is the $r$-sifted part. Doing the same procedure for $n+h_i$, $i=1,\dots,k$, we have that $a_ib_i=a_0b_0+h_i$ and, in order for this system to be solvable, it is necessary that $(a_i,a_j)|(h_i-h_j)$ for all $i\neq j$, recalling that $h_0=0$. Then:
\[
\sum_{\substack{{n\leq x}\\{\prod_{\substack{{p^\alpha||(n+h_i)}\\{p\leq r}}}p^\alpha}\leq r^{A_x}}}\lambda_r(n)\lambda_r(n+h_1)\cdots\lambda_r(n+h_k)
=
\sum_{\substack{{a_0,a_1,\dots,a_k\leq r^{A_x}}\\{p|a_i\Rightarrow p\leq r}\\{(a_i,a_j)|(h_i-h_j)}}}
\lambda(a_0)\lambda(a_1)\cdots\lambda(a_k)
\sum_{\substack{{b_0\leq x/a_0}\\{a_ib_i=a_0b_0+h_i}\\{p|b_i\Rightarrow p>r}}}
\chi(b_0)\chi(b_1)\cdots\chi(b_k),
\]
which follows from the definition of $\lambda_r$ and after writing each $n+h_i$ as a product of its $r$-smooth and its $r$-sifted parts.

The key now is to parametrize the $b_i$'s and to notice that if the system 
\begin{align*}
        \left\{\def\arraystretch{1.2}%
        \begin{array}{@{}l@{\quad}l@{}}
            a_1b_1 = a_0b_0 + h_1,\\
            \vdots\\
            a_kb_k = a_0b_0 + h_k,
        \end{array}\right.
\end{align*}
is solvable in the unknowns $b_0,b_1,\dots,b_k$, then the solutions are given by
\[
b_i = b_i^* + m\frac{\lcm(a_0,a_1,\dots,a_k)}{a_i} =: b_i^*+ma_i^*,
\]
where $(b_0^*,b_1^*,\dots,b_k^*)$ is one particular solution to the system and where $m$ ranges over all integers; see Appendix \ref{Parametrization}. Furthermore, we can take the $b_i^*$'s to be positive and minimal, in the sense that $b_i^*> 0$ for all $i$, while $b_i^*-a_i^*<0$ for at least one $i$ (this allows us to restrict ourselves to non-negative integers $m$ and makes it so that $0< b_i^*<a_i^*$ for at least one $i$, both of which are required to be completely rigorous in the next step). Then, using the fact that $\chi$ is periodic modulo $q$, we can write our sum as 
\begin{align}
\label{main_term}
&\sum_{\substack{{a_0,a_1,\dots,a_k\leq r^{A_x}}\\{p|a_i\Rightarrow p\leq r}\\{(a_i,a_j)|(h_i-h_j)}}}
\lambda(a_0)\lambda(a_1)\cdots\lambda(a_k)
\sum_{\substack{{b_0\leq x/a_0}\\{a_ib_i=a_0b_0+h_i}\\{p|b_i\rightarrow p>r}}}
\chi(b_0)\chi(b_1)\cdots\chi(b_k)\\
&=
\sideset{}{^*}\sum_{\substack{{a_0,a_1,\dots,a_k\leq r^{A_x}}\\{p|a_i\Rightarrow p\leq r}\\{(a_i,a_j)|(h_i-h_j)}}}
\lambda(a_0)\lambda(a_1)\cdots\lambda(a_k)
\sum_{n=0}^{q-1}
\chi(b_0^*+na_0^*)\chi(b_1^*+na_1^*)\cdots\chi(b_k+na_k^*)\\
&\quad\quad\quad\times
\left(\#\left\{m\leq \frac{x}{q\lcm(a_0,a_1,\dots,a_k)}: \left(\prod_{i=0}^k(b_i^*+na_i^*+mqa_i^*),\prod_{p\leq r}p\right) = 1 \right\} + \mathcal{O}_k(1)\right),
\end{align}
where the last factor counts the number of solutions which fall into each congruence class modulo $q$ and where $\sideset{}{^*}\sum$ indicates that we are only summing over the $a_i$'s for which the system is solvable. The conditions on $a_i$ which make the system solvable are determined by the Smith Normal Form of the system; see Appendix \ref{Parametrization}. For our purposes, we only care about the necessary condition $(a_i,a_j)|(h_i-h_j)$: this will allow us to control $\lcm(a_0,a_1,\dots,a_k)$, which will be needed later on in the proof. Note: the Big-O term comes from the fact that we are really looking at $b_0=b_0^*+na_0^*+mqa_0^*\leq x/a_0$ (so that $m\leq x/(qa_0a_0^*) - b_0^*/qa_0^* - n/q$, with $0\leq n\leq q-1$).

Now, the Fundamental Lemma of Sieve Theorem (FLST, Lemma \ref{FLST}) can be used to count the number of solutions which fall into each congruence class. The starting point for this is to get an asymptotic estimate for the number of such $m$ which fall into the arithmetic progression $0\Mod{d}$ for $d|\prod_{p\leq r}p$. So, let 
\[
\nu(d):=\#\left\{m\in\Z/d\Z:\prod_{i=0}^k(b_i^*+na_i^*+mqa_i^*)\equiv 0 \Mod{d}\right\}
\]
and note that
\begin{align*}
    \#\left\{m\leq \frac{x}{q\lcm(a_0,a_1,\dots,a_k)}: \prod_{i=0}^k(b_i^*+na_i^*+mqa_i^*)\equiv 0\Mod{d} \right\}
    &=
    \frac{x}{q\lcm(a_0,a_1,\dots,a_k)}\frac{\nu(d)}{d} + \mathcal{O}(\nu(d)).
\end{align*}
By the Chinese Remainder Theorem, $\nu(d)$ is completely determined by $\nu(p)$, for $p\leq r$. Moreover, $\nu(p)$ is equal to the number of distinct solutions $m\Mod{p}$ to any of the following linear congruences:
\begin{align*}
        \left\{\def\arraystretch{1.2}%
        \begin{array}{@{}l@{\quad}l@{}}
            (qa_0^*)m\equiv -(b_0^*+na_0^*)\Mod{p}\\
            (qa_1^*)m\equiv -(b_1^*+na_1^*)\Mod{p}\\
            \vdots\\
            (qa_k^*)m\equiv -(b_k^*+na_k^*)\Mod{p}.
        \end{array}\right.
\end{align*}
In order to obtain precise estimates for $\nu(p)$, we consider various possibilities depending on whether or not $qa_i^*$ is invertible modulo $p$. For starters, we restrict the sum over $n$ to the sum over $n$ such that $(b_i^*+na_i^*,q)=1$ for all $i$; otherwise, $\chi(b_i^*+na_i^*)=0$, which contributes nothing to Equation \ref{main_term}. We consider the following scenarios:
\begin{enumerate}
    \item $p|q$: In the case where $p|q$, there are no solutions because this would require that $p|b_i^*+na_i^*$ for at least one $i$, contrary to our hypothesis that $(b_i^*+na_i^*,q)=1$ for all $i$; i.e., $\nu(p)=0$ when $p|q$.
    \item $p\nmid qa_0a_1\dots a_k$: In the case where $p\nmid q$ and $p\nmid a_i$ for any $i$, we have that $qa_i^*$ is invertible modulo $p$ for all $i$; in particular, each equation in the system produces exactly one solution. If, in addition, we have that $p>\max_i\{h_i\}$, then we have that $\nu(p)=k+1$. To see this, let $m_i:=-(qa_i^*)^{-1}(b_i^*+na_i^*)\Mod{p}$ for all $i$ and note that $m_i=m_j$ iff $a_i^*b_j^*=a_j^*b_i^*\Mod{p}$. Then, using the fact that $a_i^*=\lcm(a_0,a_1,\dots,a_k)/a_i$, together with the fact that both $\lcm(a_0,a_1,\dots,a_k)$ and $a_i$ are invertible $\Mod{p}$, we have that $m_i=m_j$ iff $a_ib_i^*=a_jb_j^*\Mod{p}$. Then, recalling our definition of the $b_i^*$'s, we have that $m_i=m_j$ iff $p|(h_i-h_j).$ Since $1\leq |h_i-h_j| \leq \max_i\{h_i\}$, it follows that the $m_i$'s are distinct for $p>\max_i\{h_i\}$, which implies that $\nu(p)=k+1$ for such $p$. In other words, the $a_i$'s are relatively prime on the primes $p>\max_i\{h_i\}$ and this makes it so that there are exactly $k+1$ solutions to our system of linear congruences, provided $p$ is large enough. For the smaller primes, we content ourselves with the fact that $\nu(p)\leq p$.
    \item $p\nmid q, p|a_i$ (for at least one $i$): As mentioned above, the $a_i$'s are pairwise relatively prime on the primes $p>\max_i\{h_i\}$; that is, if $p|a_i$ for at least one $i$ and if $p>\max_i\{h_i\}$, then $p|a_i$ for exactly one $i$, say $a_{i_0}$. Moreover, this implies that $p\nmid a_{i_0}^*$ and that $p|a_i^*$ for all $i\neq i_0$, which again follows from the fact that the $a_i$'s are pairwise relatively prime on the primes $>\max_i\{h_i\}$. In particular, there is exactly one solution for the $i_0$-th equation (as $a_{i_0}^*$ is invertible modulo ${p}$) and there are either no solutions for the other equations or $p$ solutions, depending on whether or not $p|b_i^*$ for some $i\neq i_0$. To summarize the case where $p\nmid q,p|a_{i_0}$, we have precisely one of the following: (for $p>\max_i\{h_i\}$) either $\nu(p)=1$ or $\nu(p)=p$ and these situations occur if $p\nmid b_i^*$ for all $i\neq i_0$ or $p|b_i^*$ for some $i\neq i_0$, respectively.
\end{enumerate}

There are a few key points to notice from the above analysis. First, note that $\nu(p)$ is independent of $n$: this is clear if $p>\max_i\{h_i\}$ or if $p|q$, but even in the case where $p\leq \max_i\{h_i\}$ and $p\nmid q$, we either have that $p|a_i^*$ for some $i$ (so that we either have no solutions or $p$ solutions for the $i$-th equation) or $p\nmid a_i^*$ for some $i$ (in which case, two solutions $m_i,m_j$ are equal iff $a_i^*b_j^*=a_j^*b_i^*$, so that $\nu(p)$ is still independent of $n$). Next, we can simply restrict the sum over the $a_i$'s so that $\nu(p)=\nu(p;a_0,\dots,a_k)\neq p$ for any $p$; in the case where $\nu(p)=p$, the sum over the $b_i$'s is $0$, as all $s$ are such that $(qa_i^*)s\equiv -(b_i^*+na_i^*)\Mod{p}$, for some $i$ (which implies that there is no $m$ such that $\left(\prod_{i=0}^k(b_i^*+na_i^*+mqa_i^*),\prod_{p\leq r}p\right)=1$), and there is nothing to prove. 

After verifying that $\nu(p)$ satisfies Axioms \ref{Axiom_1} and \ref{Axiom_2}, the FLST (Lemma \ref{FLST}) then yields the following:

\begin{align*}
&\#\left\{m\leq \frac{x}{q\lcm(a_0,a_1,\dots,a_k)}: \left(\prod_{i=0}^k(b_i^*+na_i^*+mqa_i^*),\prod_{p\leq r}p\right) = 1 \right\}\\
&\quad\quad\quad\quad\quad\quad=
(1+\mathcal{O}_{k}(u^{-u/2}))
\frac{x}{q\lcm(a_0,a_1,\dots,a_k)}
\prod_{\substack{{p\leq r}\\{p\nmid q}}}\left(1-\frac{\nu(p)}{p}\right)
+
\mathcal{O}\left(\sum_{\substack{{d\leq r^u}\\{d|\prod_{p\leq r}p}}}\nu(d)\right),
\end{align*}uniformly for $u\geq 1$.

Plugging the above back into Equation \ref{main_term}, we are left to bound
\begin{align*}
&\sideset{}{^*}\sum_{\substack{{a_0,a_1,\dots,a_k\leq r^{A_x}}\\{p|a_i\Rightarrow p\leq r}\\{(a_i,a_j)|(h_i-h_j)}}}
\lambda(a_0)\lambda(a_1)\cdots\lambda(a_k)
\sum_{n=0}^{q-1}
\chi(b_0^*+na_0^*)\chi(b_1^*+na_1^*)\cdots\chi(b_k+na_k^*)\\
&\quad\quad\quad\quad\quad\quad\quad\times
(1+\mathcal{O}_{k}(u^{-u/2}))
\frac{x}{q\lcm(a_0,a_1,\dots,a_k)}
\prod_{\substack{{p\leq r}\\{p\nmid q}}}\left(1-\frac{\nu(p)}{p}\right)
+
\mathcal{O}\left(\sum_{\substack{{d\leq r^u}\\{d|\prod_{p\leq r}p}}}\nu(d)\right),
\end{align*}
which we break into three parts, according to the three summands in the last factor.

\begin{subsection}{Bounding the error term containing the sum over $d\leq r^u$.}
\label{Error-1}
Our goal in this subsection is to choose $u$ (the level of distribution) as large as possible, while minimizing the "error" term containing the sum over $d\leq r^u$. To begin, recall that $\nu(p)\leq \min\{k+1,p\}$; in particular, $\nu(d)\leq (k+1)^{\omega(d)}$ for all $d|\prod_{p\leq r},$ where $\omega(d)$ counts the number of prime divisors of $d$. If we let $\tau_{\kappa}$ denote the $\kappa$-th divisor function (which counts the number of representations an integer has as a product of $\kappa$ integers), it is clear that $(k+1)\leq \tau_{k+1}(p)$ for all $p$, so that $(k+1)^{\omega(d)}\leq \tau_{k+1}(d)$. Hence,
\begin{align*}
\sum_{\substack{{d\leq r^u}\\{d|\prod_{p\leq r}p}}}\nu(d)
&\leq
\sum_{d\leq r^u}\tau_{k+1}(d)
\ll_k
r^u(u\log r)^{k+1},
\end{align*}
which follows from the average order of $\tau_{\kappa}$; see Exercise 3.10 in \cite{Dimitris}. The total contribution to Equation (\ref{main_term}) is then bounded by 
\begin{align*}
    \ll_k
    qr^{(k+1)A_x}\cdot r^u(u\log r)^{k+1}
    \ll
    x^{1/2},
\end{align*}
as $x\rightarrow\infty$, provided $A_x\leq ((\log\log\eta)^{1/2}(\log \eta)^{1/12})/(10(k+1))$ and $u\leq ((\log\log\eta)^{1/2}(\log \eta)^{1/12})/10$, say, recalling that $r=x^{1/((\log\log\eta)^{1/2}(\log \eta)^{1/12})}$ with $q^{10}\leq x$.
\end{subsection}

\begin{subsection}{Bounding the error from the "main term" in the FLST}
\label{Error-2}
For the main term, we use Lemma \ref{CharacterSum}, with $f(x) = (b_0^*+a_0^*x)(b_1^*+a_1^*x)\cdots(b_k+a_k^*x)$, to bound the character sum: 
\begin{align*}
    \sum_{n=0}^{q-1}
    \chi(b_0^*+na_0^*)\chi(b_1^*+na_1^*)\cdots\chi(b_k^*+na_k^*)
    \ll_k
    q^{1/2+\epsilon};
\end{align*}
the key is to note that $f$ is not a square modulo $p$ for any prime $p>\max_i\{h_i\}$, which follows from the fact that $a_j^*b_i^*=a_i^*b_j^*$ iff $p|(h_i-h_j)$. Hence,
\begin{align*}
&\frac{x}{q}
\sideset{}{^*}\sum_{\substack{{a_0,a_1,\dots,a_k\leq r^{A_x}}\\{p|a_i\Rightarrow p\leq r}\\{(a_i,a_j)|(h_i-h_j)}}}
\frac{\lambda(a_0)\lambda(a_1)\cdots\lambda(a_k)}{\lcm(a_0,a_1,\dots,a_k)}
\prod_{\substack{{p\leq r}\\{p\nmid q}}}\left(1-\frac{\nu(p)}{p}\right)
\sum_{n=0}^{q-1}
\chi(b_0^*+na_0^*)\chi(b_1^*+na_1^*)\cdots\chi(b_k+na_k^*)\\
&\ll_k
\frac{x}{q^{1/2-\epsilon}}
\sum_{\substack{{a_0,a_1,\dots,a_k\leq r^{A_x}}\\{p|a_i\Rightarrow\leq r}\\{(a_i,a_j)|i-j}}}
\frac{1}{\lcm(a_0,a_1,\dots,a_k)}
\prod_{\substack{{\max_i\{h_i\}<p\leq r}\\{p\nmid q}\\{p\nmid a_0a_1\cdots a_k}}}\left(1-\frac{k+1}{p}\right)
\prod_{\substack{{\max_i\{h_i\}<p\leq r}\\{p\nmid q}\\{p|a_0a_1\cdots a_k}}}\left(1-\frac{1}{p}\right),
\end{align*}
where we have removed the contribution from the small prime factors (which is $\ll_k 1$), where we have bounded the sum over the $a_i$'s trivially, and where we have split the product over the primes according to the value of $\nu(p)$.

\begin{rem}
Here is a more detailed approach to showing that $f$ is not a square modulo $p$ for all but finitely-many $p$: we need to consider two cases, depending on whether or not $p|a_i$, for some $i$. So, suppose $p>\max_i\{h_i\}$ and that $p\nmid a_i$ for any $i$; then  $p\nmid a_i^*$ for any $i$, so that $a_i^*$ is invertible modulo $p$ for all $i$ and that $f$ can only be a square if $b_i^*(a_i^*)^{-1}\equiv b_j^*(a_j^*)^{-1}\Mod{p}$ for at least some $i\neq j$, but this occurs, as we say before, iff $p|(h_i-h_j)$ which cannot occur for $p>\max_i\{h_i\}$. In the case where $p|a_i$ for some $i$, then $p|a_i$ for exactly one $i$ (otherwise, by the condition that $(a_i,a_j)|(h_i-h_j)$, we would get a contradiction); in particular, $p|a_j^*$ for all $j\neq i$, so that $f(x)\equiv c(b_i^*+a_i^*x)\Mod{p},$ for some $c\in\Z/p\Z$ and where $p\nmid a_i^*$ and $f$ is clearly not a square.
\end{rem}

Next, using the fact that
\[
\lcm(a_0,a_1,\dots,a_k)\geq \prod_{i<j}(a_i,a_j)^{-1} a_0a_1\cdots a_k,
\]
together with the fact that $(a_i,a_j)|(h_i-h_j)$, we can bound our sum by
\begin{align*}
\ll_k
\frac{x}{q^{1/2-\epsilon}}
\sum_{\substack{{a_0,a_1,\dots,a_k\leq r^{A_x}}\\{p|a_i\Rightarrow\leq r}\\{(a_i,a_j)|(h_i-h_j)}}}
\frac{1}{a_0a_1\cdots a_k}
\prod_{\substack{{\max_i\{h_i\}<p\leq r}\\{p\nmid q}\\{p\nmid a_0a_1\cdots a_k}}}\left(1-\frac{k+1}{p}\right)
\prod_{\substack{{\max_i\{h_i\}<p\leq r}\\{p\nmid q}\\{p|a_0a_1\cdots a_k}}}\left(1-\frac{1}{p}\right).
\end{align*}
Again recalling that the $a_i$'s are pairwise relatively prime on the primes $p>\max_i\{h_i\}$, we can actually separate the variables in our sum, so that
\begin{align*}
&\frac{1}{a_0a_1\cdots a_k}
\prod_{\substack{{\max_i\{h_i\}<p\leq r}\\{p\nmid q}\\{p\nmid a_0a_1\cdots a_k}}}\left(1-\frac{k+1}{p}\right)
\prod_{\substack{{\max_i\{h_i\}<p\leq r}\\{p\nmid q}\\{p|a_0a_1\cdots a_k}}}\left(1-\frac{1}{p}\right)\\
&\quad\quad\quad\quad\quad\quad=
\prod_{i=0}^k
\frac{1}{a_i}
\prod_{\substack{{\max_i\{h_i\}<p\leq r}\\{p\nmid q}\\{p\nmid a_i}}}\left(1-\frac{k+1}{p}\right)
\prod_{\substack{{\max_i\{h_i\}<p\leq r}\\{p\nmid q}\\{p|a_i}}}\left(1-\frac{1}{p}\right).
\end{align*}

To deal with the products over the primes, we can multiply and divide by "complements" of the given divisibility conditions, the idea being that we want to "complete" our products to keep only divisibility conditions (as opposed to saying $p\nmid q$, for example):
\begin{align*}
&\prod_{\substack{{\max_i\{h_i\}<p\leq r}\\{p\nmid q}\\{p\nmid a_i}}}\left(1-\frac{k+1}{p}\right)
\prod_{\substack{{\max_i\{h_i\}<p\leq r}\\{p\nmid q}\\{p|a_i}}}\left(1-\frac{1}{p}\right)\\
&=
\prod_{{\max_i\{h_i\}<p\leq r}}\left(1-\frac{k+1}{p}\right)
\prod_{\substack{{\max_i\{h_i\}<p\leq r}\\{p| q}}}\left(1-\frac{k+1}{p}\right)^{-1}
\prod_{\substack{{\max_i\{h_i\}<p\leq r}\\{p| a_i}}}\left(1-\frac{k+1}{p}\right)^{-1}
\prod_{\substack{{\max_i\{h_i\}<p\leq r}\\{p| q}\\{p| a_i}}}\left(1-\frac{k+1}{p}\right)\\
&\quad\quad\quad\quad\quad\quad\times
\prod_{\substack{{\max_i\{h_i\}<p\leq r}\\{p|a_i}}}\left(1-\frac{1}{p}\right)
\prod_{\substack{{\max_i\{h_i\}<p\leq r}\\{p| q}\\{p|a_i}}}\left(1-\frac{1}{p}\right)^{-1}\\
&\leq
\prod_{{\max_i\{h_i\}<p\leq r}}\left(1-\frac{k+1}{p}\right)
\prod_{\substack{{\max_i\{h_i\}<p\leq r}\\{p| q}}}\left(1-\frac{k+1}{p}\right)^{-1}
\prod_{\substack{{\max_i\{h_i\}<p\leq r}\\{p| a_i}}}\left(\frac{p-1}{p-(k+1)}\right),
\end{align*}
where we have combined the factors for $p|a_i$ and where we have trivially bounded the product of the factors over $p|q,a_i$ by $\leq 1$.

Putting all of this together, we have that
\begin{align*}
&\sum_{\substack{{a_0,a_1,\dots,a_k\leq r^{A_x}}\\{p|a_i\Rightarrow\leq r}\\{(a_i,a_j)|(h_i-h_j)}}}
\frac{1}{a_0a_1\cdots a_k}
\prod_{\substack{{\max_i\{h_i\}<p\leq r}\\{p\nmid q}\\{p\nmid a_0a_1\cdots a_k}}}\left(1-\frac{k+1}{p}\right)
\prod_{\substack{{\max_i\{h_i\}<p\leq r}\\{p\nmid q}\\{p|a_0a_1\cdots a_k}}}\left(1-\frac{1}{p}\right)\\
&\ll
\prod_{{\max_i\{h_i\}<p\leq r}}\left(1-\frac{k+1}{p}\right)
\prod_{\substack{{\max_i\{h_i\}<p\leq r}\\{p| q}}}\left(1-\frac{k+1}{p}\right)^{-1}
\left(
\sum_{\substack{{a\leq r^{A_x}}\\{p|a\Rightarrow p\leq r}}}
\frac{1}{a}\prod_{\substack{{\max_i\{h_i\}<p\leq r}\\{p| a}}}\left(\frac{p-1}{p-(k+1)}\right)\right)^{k+1}.
\end{align*}
\end{subsection}

Finally, we can bound the sum over $a$ by an Euler product, as the corresponding summand is a multiplicative function supported over the $r$-smooth integers:
\begin{align*}
\sum_{\substack{{a\leq r^{A_x}}\\{p|a\Rightarrow p\leq r}}}
\frac{1}{a}\prod_{\substack{{\max_i\{h_i\}<p\leq r}\\{p| a}}}\left(\frac{p-1}{p-(k+1)}\right)
&\leq
\prod_{p\leq k+1}\left(1-\frac{1}{p}\right)^{-1}
\prod_{\max_i\{h_i\}<p\leq r}\left(1+\frac{p-1}{p-(k+1)}\left(\frac{1}{p}+\frac{1}{p^2}+\dots \right)\right)\\
&\ll_k
\prod_{\max_i\{h_i\}<p\leq r}\left(1+\frac{1}{p-(k+1)}\right).
\end{align*}
Therefore, the total contribution from the "main term" in the FLST is bounded above by
\begin{align*}
    &\ll_k
    \frac{x}{q^{1/2-\epsilon}}
    \prod_{{\max_i\{h_i\}<p\leq r}}\left(1-\frac{k+1}{p}\right)
\prod_{\substack{{\max_i\{h_i\}<p\leq r}\\{p| q}}}\left(1-\frac{k+1}{p}\right)^{-1}
\prod_{\max_i\{h_i\}<p\leq r}\left(1+\frac{1}{p-(k+1)}\right)^{k+1}\\
&\ll_k
    \frac{x}{q^{1/2-\epsilon}}
    \prod_{\substack{{\max_i\{h_i\}<p\leq r}\\{p| q}}}\left(1-\frac{k+1}{p}\right)^{-1},
\end{align*}
where the last line follows from the fact that
\begin{align*}
    \prod_{{\max_i\{h_i\}<p\leq r}}\left(1-\frac{k+1}{p}\right)\left(1+\frac{1}{p-(k+1)}\right)^{k+1}
    &=
    \prod_{\max_i\{h_i\}<p\leq r} \left(1+\mathcal{O}\left(\frac{1}{p^2}\right)\right)=\mathcal{O}(1),
\end{align*}
which in turn follows from the Binomial Theorem. Hence, the main term is bounded above by
\begin{align*}
    \frac{x}{q^{1/2-\epsilon}}    \prod_{\substack{{\max_i\{h_i\}<p\leq r}\\{p| q}}}\left(1-\frac{k+1}{p}\right)^{-1}
    &\ll_k
    \frac{x\log^{k+1}q}{q^{1/2-\epsilon}},
\end{align*}
where the first line follows by taking exp-log of the product and using the Taylor series expansion of $\log$ and which produces an admissible error, after recalling that $x\leq q^{(\log\log\eta)/3}$.

\begin{subsection}{Dealing with the error from $u^{-u/2}$.} 
\label{Error-3}
For the error containing the term $u^{-u/2}$, we have that Equation (\ref{main_term}) is bounded above by
\begin{align*}
    &\ll_k
    u^{-u/2}\frac{x}{q}
\sum_{\substack{{a_0,a_1,\dots,a_k\leq r^{A_x}}\\{p|a_i\Rightarrow p\leq r}\\{(a_i,a_j)|(h_i-h_j)}}}
\frac{1}{\lcm(a_0,a_1,\dots,a_k)}
\sum_{n=0}^{q-1}
|\chi(b_0^*+na_0^*)\chi(b_1^*+na_1^*)\cdots\chi(b_k^*+na_k^*)|
\prod_{\substack{{p\leq r}\\{p\nmid q}}}\left(1-\frac{\nu(p)}{p}\right)\\
&=
u^{-u/2}\frac{x}{q}
\sum_{\substack{{a_0,a_1,\dots,a_k\leq r^{A_x}}\\{p|a_i\Rightarrow p\leq r}\\{(a_i,a_j)|(h_i-h_j)}}}
\frac{1}{\lcm(a_0,a_1,\dots,a_k)}
\prod_{\substack{{p\leq r}\\{p\nmid q}}}\left(1-\frac{\nu(p)}{p}\right)
\sum_{\substack{{n=0}\\{\left(\prod_{i=0}^{k}(b_i^*+na_i^*\right),q)=1}}}^{q-1}1.\\
\end{align*}

To estimate the sum over $n$, we use the fact that 
\[
\sum_{d|n}\mu(n)
=
        \left\{\def\arraystretch{1.2}%
        \begin{array}{@{}l@{\quad}l@{}}
            \text{$1$ if $n=1$}\\
            \text{$0$ otherwise},
        \end{array}\right.
\]
which yields:
\begin{align*}
    \sum_{\substack{{n=0}\\{\left(\prod_{i=0}^{k}(b_i^*+na_i^*\right),q)=1}}}^{q-1}1
    &=
    \sum_{\substack{{n=0}}}^{q-1}\sum_{d|\left(\prod_i (b_i^*+na_i^*),q\right)}\mu(d)\\
    &=
    \sum_{d|q}\sum_{\substack{{n=0}\\{d|\prod_i(b_i^*+na_i^*)}}}^{q-1}\mu(d)\\
    &=
    q\sum_{d|q}\frac{N(d)}{d},
\end{align*}
where $N(d)$ is the number of $n\in\Z/d\Z$ such that $\prod_{i=0}^{k+1}(b_i^*+na_i^*)\equiv 0\Mod{d},$ $d|q$. By the CRT, $N(d)$ is multiplicative, so that the sum can be written as 
\begin{align*}
        q\sum_{d|q}\frac{\mu(d)N(d)}{d}
        &=
        q\prod_{p|q}\left(1-\frac{N(p)}{p}\right)
\end{align*}
and it remains to compute $N(p)$; we consider various cases, depending on whether or not $a_i^*$ is invertible $\Mod{p}$, noting that $N(p)$ is equal to the number of $n \Mod{p}$ such that $na_i^*\equiv -b_i^*\Mod{p}$, for any $i$.
\begin{enumerate}
    \item $p\nmid a_0a_1\cdots a_k$: In this case, all the $a_i^*$ are invertible modulo $p$ so that exactly one $n$ satisfies the given congruence in the $i$-th equation. Assuming further that $p>\max_i\{h_i\}$, we get $k+1$ distinct solutions as $(a_i^*)^{-1}b_i^*=(a_j^*)^{-1}b_j^*\Mod{p}$ iff $p|(h_i-h_j)$; that is, $N(p)=k+1$ if $p>\max_i\{h_i\}$, with $1\leq N(p)\leq p$ otherwise.
    \item $p|a_0a_1\cdots a_k:$ Similarly, $N(p)=1$ if $p>\max_i\{h_i\}$ with $N(p)\leq p$ otherwise. The idea here is that the $a_i$'s are pairwise relatively prime on the primes $p>\max_i\{h_i\}$ so that exactly one $a_i$ is divisible by $p$ if $p|a_0\cdots a_k$, say $a_{i_0}$; in particular, we get exactly one solution at level $i_0$ and the other congruences are solvable iff $b_i^*\equiv 0\Mod{p}$ for some $i$. In the latter case, we have that $N(p)=p$ and the product is $0$, so we may assume that $p\nmid b_i^*$ for all $i$.
\end{enumerate}

Thus, the second error term is bounded above by
\begin{align*}
&\ll_k
u^{-u/2}x
\sum_{\substack{{a_0,a_1,\dots,a_k\leq r^{A_x}}\\{p|a_i\Rightarrow p\leq r}\\{(a_i,a_j)|(h_i-h_j)}}}
\frac{1}{\lcm(a_0,a_1,\dots,a_k)}\\
&\quad\quad\quad\times
\prod_{\substack{{\max_i\{h_i\}<p\leq r}\\{p\nmid q}\\{p\nmid a_0a_1\cdots a_k}}}\left(1-\frac{k+1}{p}\right)
\prod_{\substack{{\max_i\{h_i\}<p\leq r}\\{p\nmid q}\\{p|a_0a_1\cdots a_k}}}\left(1-\frac{1}{p}\right)
\prod_{\substack{{\max_i\{h_i\}<p}\\{p|q}\\{p\nmid a_0a_1\cdots a_k}}}\left(1-\frac{k+1}{p}\right)
\prod_{\substack{{\max_i\{h_i\}<p}\\{p|q}\\{p|a_0a_1\cdots a_k}}}\left(1-\frac{1}{p}\right)\\
&\ll_k
u^{-u/2}x
\sum_{\substack{{a_0,a_1,\dots,a_k\leq r^{A_x}}\\{p|a_i\Rightarrow p\leq r}\\{(a_i,a_j)|(h_i-h_j)}}}
\frac{1}{a_0a_1\cdots a_k}
\prod_{\substack{{\max_i\{h_i\}<p\leq r}\\{p\nmid a_0a_1\cdots a_k}}}\left(1-\frac{k+1}{p}\right)
\prod_{\substack{{\max_i\{h_i\}<p\leq r}\\{p|a_0a_1\cdots a_k}}}\left(1-\frac{1}{p}\right)\\
&\ll_k
u^{-u/2}x
\prod_{\substack{{\max_i\{h_i\}<p\leq r}}}\left(1-\frac{k+1}{p}\right)
\sum_{\substack{{a_0,a_1,\dots,a_k\leq r^{A_x}}\\{p|a_i\Rightarrow p\leq r}\\{(a_i,a_j)|(h_i-h_j)}}}
\frac{1}{a_0a_1\cdots a_k}
\prod_{\substack{{\max_i\{h_i\}<p\leq r}\\{p| a_0a_1\cdots a_k}}}\left(1-\frac{k+1}{p}\right)^{-1}
\prod_{\substack{{\max_i\{h_i\}<p\leq r}\\{p|a_0a_1\cdots a_k}}}\left(1-\frac{1}{p}\right).
\end{align*}

Again using the fact that the $a_i$'s are pairwise relatively on the primes $p>\max_i\{h_i\}$, we can separate the variables in the sum over the $a_i$'s:
\begin{align*}
&\sum_{\substack{{a_0,a_1,\dots,a_k\leq r^{A_x}}\\{p|a_i\Rightarrow p\leq r}\\{(a_i,a_j)|(h_i-h_j)}}}
\frac{1}{a_0a_1\cdots a_k}
\prod_{\substack{{\max_i\{h_i\}<p\leq r}\\{p| a_0a_1\cdots a_k}}}\left(1-\frac{k+1}{p}\right)^{-1}
\prod_{\substack{{\max_i\{h_i\}<p\leq r}\\{p|a_0a_1\cdots a_k}}}\left(1-\frac{1}{p}\right)\\
&=
\sum_{\substack{{a_0,a_1,\dots,a_k\leq r^{A_x}}\\{p|a_i\Rightarrow p\leq r}\\{(a_i,a_j)|(h_i-h_j)}}}
\left(\prod_{i=0}^k
\frac{1}{a_i}
\prod_{\substack{{\max_i\{h_i\}<p\leq r}\\{p| a_i}}}\left(1-\frac{k+1}{p}\right)^{-1}\left(1-\frac{1}{p}\right)
\right)\\
&\leq
\left(\sum_{\substack{{a\leq r^{A_x}}\\{p|a\Rightarrow p\leq r}}}
\frac{1}{a}
\prod_{\substack{{\max_i\{h_i\}<p\leq r}\\{p| a}}}\left(\frac{p-1}{p-(k+1)}\right)\right)^{k+1}\\
&\leq
\left(\prod_{p\leq \max_i\{h_i\}}\left(1-\frac{1}{p}\right)^{-1}\prod_{\max_i\{h_i\}<p\leq r}\left(1+\frac{p-1}{p-(k+1)}\left(\frac{1}{p}+\frac{1}{p^2}+\dots\right)\right)\right)^{k+1}.
\end{align*}

Therefore, the total contribution from the secondary error term is
\begin{align*}
    &\ll_k
    u^{-u/2}x
    \prod_{\max_i\{h_i\}<p\leq r}\left(1-\frac{k+1}{p}\right)\left(1+\frac{1}{p-(k+1)}\right)^{k+1}\\
    &\ll_k u^{-u/2}x,
\end{align*}
noting that this is the same product over the primes that we encountered earlier and which we saw was $\ll_k 1$. Taking $u\asymp (\log\log\eta)^{1/2}(\log\eta)^{1/12}$ produces an admissible error and thus establishes Theorem \ref{main_thm}.
\end{subsection}

\subsection{Fitting the pieces.} In this section, we want to say a few words which justify our choice of parameters for $r, A_x,$ and $u$. Setting $r=x^{1/\alpha}$, we need to choose the largest possible $A_x$ and $u$ for which the following errors are minimized, while optimizing $\alpha$:
\begin{align*}
\left\{\def\arraystretch{1.2}%
    \begin{array}{@{}l@{\quad}l@{}}
        \frac{x}{\log r}\left(\exp\left(\frac{\log x}{\log q}\right)(\log q)(\log \eta)^{-1/2}) + \log\log q\right) + \frac{x}{r} \quad\mbox{(error from the proof of Lemma \ref{lambda_r})},\\
        \frac{x}{A_x} \quad\mbox{(error from Lemma \ref{Pre-Sieve})},\\
        qr^{(k+1)A_x+u}(u\log r)^{k+1} \quad\mbox{(error from Section \ref{Error-1})},\\
        \frac{x\log^{k+1}q}{q^{1/2-\epsilon}} \quad\mbox{(error from Section \ref{Error-2})},\\
        xu^{-u/2} \quad\mbox{(error from Section \ref{Error-3})}.
    \end{array}\right.    
\end{align*}
For the error from Section \ref{Error-1}, we can get power savings by taking $(k+1)A_x = u = C\alpha$, for $C>0$ sufficiently small. The error from Lemmas \ref{lambda_r} and \ref{Pre-Sieve}, after normalizing by $x$, are of the form $\alpha/ ((\log\log\eta)(\log\eta)^{1/6})$ and $1/\alpha$, respectively, which yields the optimal choice of $\alpha:= (\log\log\eta)^{1/2}(\log\eta)^{1/12},$ thus establishing Theorem \ref{main_thm}.

\section{Acknowledgements}
The author would like to thank Maksym Radziwi{\l}{\l} for suggesting this problem and for his continuous support throughout the research phase of this paper. The author would also like to thank Dimitris Koukoulopoulos, Mariusz Lemańczyk, Patrick Meisner, James Rickards, and Peter Zenz for many helpful discussions and comments. Finally, some of this work was conducted while the author was visiting CalTech; he is grateful for their hospitality.

\bibliographystyle{alpha}
\bibliography{main}

\textsc{Department of Mathematics and Statistics, McGill University, 805 Sherbrooke St. W., Montreal, QC H3A 2K6, Canada}

\textit{Email address:} \href{mailto:iakov.chinis@gmail.com}{{\texttt{iakov.chinis@gmail.com}}}

\appendix
\section{The Fundamental Lemma of Sieve Theory}
\label{Sieve_Theory}

In this section, we present the Axioms of Sieve Theory, culminating in the Fundamental Lemma of Sieve Theory (FLST). We use the ideas presented here in order to move from a character sum over $r$-sifted integers to a character sum over one complete residue class (see Section \ref{Proof}). We follow Chapters 18 and 19 of \cite{Dimitris} and start off by listing some notation and the appropriate hypotheses needed for the FLST.

Let $\mathcal{A}$ denote a finite set of integers and let $\mathcal{P}$ denote a finite set of primes. We are interested in counting the number of elements of $\mathcal{A}$ which are relatively prime to $\mathcal{P}$; that is, we are interested in bounding the quantity
\[
S(\mathcal{A},\mathcal{P}):=\#\{a\in \mathcal{A}:(a,\mathcal{P})=1\},
\]
where $(a,\mathcal{P})=1$ means that $a$ has no prime factors in $\mathcal{P}$. In order to do so, it suffices to look at elements of $a\in\mathcal{A}$ which are divisible by $d|\prod_{p\in\mathcal{P}}p$; see Examples 18.1-18.6 in \cite{Dimitris}. More precisely, we are interested in having an asymptotic estimate for 
\[
\mathcal{A}_d:=\#\{a\in\mathcal{A}:a\equiv 0\Mod{d}\},
\]
so we assume the following:

\begin{axiom}
\label{Axiom_1}
There exists a multiplicative function $\nu$, a parameter $X$, and a sequence of remainders $(r_d)_{d|\mathcal{P}}$ such that 
\[
\text{$\mathcal{A}_d=\frac{\nu(d)}{d}X+r_d$ \;\;\;for all $d|\mathcal{P}$}
\]
and
\[
\text{$\nu(p)<p$ \;\;\;for all $p\in\mathcal{P}$.}
\]
\end{axiom}

\begin{rem}
Note that $d|\mathcal{P}$ is shorthand for $d|\prod_{p\in\mathcal{P}}p$.
\end{rem}

We should think of $\nu(p)$ as the number of residue classes modulo $p$ we are "removing" in order to capture elements of our set $\mathcal{A}$ which are prime to $\mathcal{P}$. As such, we want some sort of control over $\nu(p)$:

\begin{axiom}
\label{Axiom_2}
There are constants $\kappa,k\geq 0$ and $\epsilon\in(0,1]$ such that
\[
\text{$\sum_{p\in\mathcal{P}\cap[1,\omega]}\frac{\nu(p)\log p}{p} = \kappa\log \omega + \mathcal{O}(1)$ \;\;\;for all $\omega\leq \max\mathcal{P}$}
\]
and 
\[
\text{$\nu(p)\leq \min\{k,(1-\epsilon)p\}$ \;\;\;for all $p\in \mathcal{P}$.}
\]
\end{axiom}

Assuming that Axioms \ref{Axiom_1} and \ref{Axiom_2} hold, we are able to compute $S(\mathcal{A},\mathcal{P})$:

\begin{lem}[The Fundamental Lemma of Sieve Theory]
\label{FLST}
Suppose $\mathcal{A}$ and $\mathcal{P}$ satisfy Axioms \ref{Axiom_1} and \ref{Axiom_2} for some constants $\kappa,k\geq 0$, $\epsilon\in(0,1]$, and let $y=\max\mathcal{P}$. Then:
\[
S(\mathcal{A},\mathcal{P})
=
(1+\mathcal{O}_{\kappa,k,\epsilon}(u^{-u/2}))
X
\prod_{p\in\mathcal{P}}\left(1-\frac{\nu(p)}{p}\right)
+
\mathcal{O}\left(\sum_{d\leq y^u,d|\mathcal{P}}|r_d|\right),
\]
uniformly for $u\geq 1$.
\end{lem}
\begin{proof}
See \cite[Chapters 18-19]{Dimitris}.
\end{proof}

\begin{rem}
We call $D=y^u$ the \textbf{level of distribution} of the sieve. As Koukoulopoulos remarks in his book, the level of distribution is "a measure of how well we can control the distribution of $\mathcal{A}$ among the progressions $0\Mod{d}$," with $d|\mathcal{P}$. In order to control the level of distribution, we often use a "preliminary sieve," which removes integers with smaller prime factors and then use another sieve to remove larger primes. This is exactly what we do in the proof of Theorem \ref{main_thm} by sieving out the integers whose $r$-smooth part is "large."
\end{rem}

\section{Character Sums}
\label{Character_Sums}

The key to the work of Germ{\'a}n and K{\'a}tai is to approximate the Liouville function by a real, primitive Dirichlet character $\chi \Mod{q}$ on "large" primes, so that, with the help of some sieve theory, we can change our problem of bounding $k$-point correlations of $\lambda$ to one of bounding character sums with a polynomial argument, which are well understood. For our purposes, we need to bound character sums of the form
\[
\sum_{n\Mod{q}}\chi(f(n)),
\]
where $f$ is some polynomial with integer coefficients which can be factored into distinct linear factors, with $q$ equal to the conductor/modulus of the real, primitive character $\chi$. There are various instances of these types of bounds when the conductor $q$ is a prime, dating back to the work of Weil on the Riemann Hypothesis over finite function fields; see, \cite{Burgess}, for example, and also \cite{Schmidt} for an elementary approach to understanding curves over finite fields. In \cite{Liou_Siegel}, they look at $f(x)=x(x+1),$ it which case it is known that the above character sum is exactly equal to $-1$. For general $f$, we have the following, due to Weil:

\begin{lem}[Weil]
\label{CharacterSum}
Let $\chi$ be a Dirichlet character modulo $p$ of order $d|(p-1)$. If $f\in\Z[x]$ is not a $d$-th power modulo $p$ (i.e., $f(x)\not\equiv cg(x)^d \Mod{p}$ identically for any $c\in \Z$ and any $g\in\Z[x]$) and if $f$ has $m$ distinct roots modulo $p$, then:
\[
\left|\sum_{n\Mod{p}}\chi(f(n))\right|\leq (m-1)p^{1/2},
\]
where the sum runs over an entire residue class modulo the prime $p$.
\end{lem}
\begin{proof}
See \cite[Theorem 2C' (pg. 43)]{Schmidt} (or even \cite[Theorem 11.23]{IwaKow}/\cite[Lemma 9.25]{MV}).
\end{proof}

Our goal is to apply Lemma \ref{CharacterSum} with $\chi$ a real, primitive Dirichlet character modulo $q$, with $q$ not necessarily a prime. Fortunately for us, all such characters have conductor $q=2^jm$, where $j\leq 3$ and where $m$ is an odd, squarefree integer; see \cite[Section 9.3]{MV}, for example. Furthermore, the Chinese Remainder Theorem allows us to write each $n\Mod{q}$ uniquely as
\[
n=a_1\frac{q}{p_1^{\alpha_1}}+\dots+a_s\frac{q}{p_s^{\alpha_s}},
\]
for any $q=p_1^{\alpha_1}\cdots p_s^{\alpha_s}$, with $a_i$ varying over a complete residue class modulo $p_i^{\alpha_i}$. In particular, for characters $\chi$ modulo $q$ such that $\chi=\chi_1\cdots\chi_s$ with $\chi_i$ a character modulo $p_i^{\alpha_i}$,
\begin{align*}
    \sum_{n\Mod{q}}\chi(f(n))
    &=
    \sum_{a_1,\dots,a_s}\prod_{i=1}^s\chi_i\left(f\left(a_1\frac{q}{p_1^{\alpha_1}}+\dots+a_s\frac{q}{p_s^{\alpha_s}}\right)\right)\\
    &=
    \prod_{i=1}^s \sum_{a_i\Mod{p_i^{\alpha_i}}}\chi_i\left(f\left(a_i\frac{q}{p_i^{\alpha_i}}\right)\right),
\end{align*}
where the last line follows from the fact that $\chi_i$ is periodic with period $p_i^{\alpha_i}$. Then, using the fact that every real, primitive character $\chi \Mod{q}$ can be written uniquely as $\chi = \chi_1\cdots \chi_s$ with each $\chi_i$ being real and primitive, Lemma \ref{CharacterSum} implies that
\begin{align*}
        \sum_{n\Mod{q}}\chi(f(n))
        &\ll
        (\deg(f) - 1)^sq^{1/2},
\end{align*}
provided $f$ is not a square modulo $p$ for all but finitely many primes $p$. Setting $N=\deg(f)-1$ and noting that $\omega(q)=s$, we then have that $N^{\omega(q)}\leq \tau_N(q) \ll q^\epsilon$, for any $\epsilon>0$, so that
\begin{align*}
    \sum_{n\Mod{q}}\chi(f(n))
    &\ll
    q^{1/2+\epsilon},
\end{align*}
for any real, primitive Dirichlet character modulo $q$, provided $f$ is not a square modulo $p$ for all but finitely-many $p$.

\section{Parametrization}
\label{Parametrization}

In this section, we briefly discuss how to use the Smith Normal Form of a matrix in order to solve a system of Diophantine equations. Our ultimate goal is to apply this technique in order to show that the solutions of the following system of integer equations
\begin{align*}
        \left\{\def\arraystretch{1.2}%
        \begin{array}{@{}l@{\quad}l@{}}
            a_1b_1 = a_0b_0 + h_1,\\
            \vdots\\
            a_kb_k = a_0b_0 + h_k,
        \end{array}\right.
\end{align*}
in the unknowns $b_0,b_1,\dots,b_k$, are given by
\[
b_i=b_i^* + m\frac{\lcm(a_0,a_1,\dots,a_k)}{a_i}=:b_i^*+ma_i^*,
\]
where $(b_0*,b_1^*,\dots,b_k^*)$ is one particular solution of the system and where $m$ ranges over all the integers.

\begin{rem}
In the case where $k=1$ and $(a_0,a_1)=1$, Bezout's Lemma tells us that the $b_i$'s can be parametrized as 
\[
b_i=b_i^*+m\frac{a_1a_2}{a_i},
\]
where $(b_0^*,b_1^*)$ is one particular solution of the system and where $m$ ranges over all integers. The proof of Bezout's Lemma follows from the Chinese Remainder Theorem; it is very likely that the proof generalizes, but we prefer to use a more direct method. Also, it is clear that such $b_i$'s generate a set of solutions and seems likely that one could show that all solutions must be of the form above. In any case, the SNF gives us a versatile tool to handle more general cases.
\end{rem}

Let $A$ be an $m\times n$ matrix with integer entries and consider the system $AX=C$, for a given integer matrix $C$. Then, there exist invertible matrices $U$ and $V$ with integer entries such that $B:=UAV$ is (almost) diagonal: in general, $B$ may not be a square matrix, but the non-diagonal entries will be zero. We call $B$ the \textit{Smith Normal Form} of $A$ and finding the matrices $U$ and $V$ amounts to using limited versions of the elementary row and column operations which preserve integer entries: since we are looking for invertible matrices $U$ and $V$ with integer entries, we must ensure that whatever operations we apply to the matrix $A$ will preserve our integer entries. What is important here is that solving $AX=C$ in the integers is equivalent to making the change of variable $Y=V^{-1}X$ and solving the system $BY=D$, where $D:=UC$. In particular, the original system will have integer solutions iff $b_{ii}y_i=d_i$ for all $i$ (where $D$ is a column matrix and $d_i$ is the entry in row $i$). This last system is then solvable over the integers iff $b_{ii}|d_i$ whenever $b_{ii}\neq 0$ and $d_i=0$ whenever $b_{ii}=0$, in which case,
\[
X=V
\begin{bmatrix}
\frac{b_{11}}{d_1}\\
\frac{b_{22}}{d_2}\\
\vdots\\
\frac{b_{kk}}{d_k}\\
f_{k+1}\\
\vdots\\
f_{n}
\end{bmatrix},
\]
where the $b_{ii}$'s are arranged so that $b_{ii}\neq 0$ for all $i=1,\dots,k$ and where $f_{k+1},\dots,f_n$ are arbitrary integers (representing the $n-k$ free variables). 

The above decomposition hinges on our ability to find invertible matrices $U$ and $V$ such that $UAV$ is diagonal. We illustrate how to do this for the following system of equations in the unknowns $b_0,b_1,\dots,b_k$,
\begin{align*}
        \left\{\def\arraystretch{1.2}%
        \begin{array}{@{}l@{\quad}l@{}}
            a_1b_1 = a_0b_0 + h_1,\\
            \vdots\\
            a_kb_k = a_0b_0 + h_k,
        \end{array}\right.
\end{align*}
proving each solution $(b_0,b_1,\dots,b_k)$ can be parametrized as 
\[
b_i=b_i^*+ma_i^*,
\]
where $(b_0^*,b_1^*,\dots,b_k^*)$ is one particular solution of the system (assuming that the system is solvable) and where 
\[
a_i^*:=\frac{\lcm(a_0,a_1,\dots,a_k)}{a_i},
\]
for $i=0,1,\dots,k$.

To show the above, we simply use the algorithm which produces the SNF of $A$. So, let
\[
A:=
\begin{bmatrix}
a_0 & -a_1 & 0 & \dots & 0 & 0\\
0   &  a_1 & -a_2 &\dots & 0 & 0\\
\vdots & & & & &\\
0 & 0 & 0 & \dots & a_{k-1} & -a_{k}  
\end{bmatrix},
\]
\[
X:=
\begin{bmatrix}
b_0\\
b_1\\
\vdots\\
b_{k},
\end{bmatrix}
\]
and 
\[
C:=
\begin{bmatrix}
-h_1\\
h_1-h_2\\
\vdots\\
h_{k-1}-h_{k}
\end{bmatrix}.
\]

Our goal it to solve the system $AX=C$ over the integers. To begin, let $d_{0,1}:=(a_0,a_1)$. Then, there exist integers $x_{0,1},y_{0,1}$ such that $d_{0,1}=a_0x+a_1y$; in particular, $1=x_{0,1}a_0/d_{0,1}+y_{0,1}a_1/d_{0,1}$. Embedding this information into a $(k+2)\times (k+2)$ identity matrix $V_1$, we can get the following:
\begin{align*}
AV_1
&=
\begin{bmatrix}
a_0 & -a_1 & 0 & \dots & 0 & 0\\
0   &  a_1 & -a_2 &\dots & 0 & 0\\
\vdots & & & & &\\
0 & 0 & 0 & \dots & a_{k-1} & -a_{k}  
\end{bmatrix}
\begin{bmatrix}
x_{0,1}  & a_1/d_{0,1} & 0 & \dots & 0\\
-y_{0,1} & a_0/d_{0,1} & 0 & \dots & 0\\
0 & 0 & 1 & \dots & 0\\
\vdots & & & &\\
0 & 0 & 0 & \dots & 1  
\end{bmatrix}\\
&=
\begin{bmatrix}
d_{0,1} & 0 & 0 & \dots & 0 & 0\\
-a_1y   &  a_0a_1/d_{0,1} & -a_2 &\dots & 0 & 0\\
\vdots & & & & &\\
0 & 0 & 0 & \dots & a_{k-1} & -a_{k}  
\end{bmatrix}.
\end{align*}

To get zeros below the leading variable, consider the $(k+1)\times (k+1)$ matrix $U_1$ defined by
\[
U_1
:=
\begin{bmatrix}
1  & 0 & 0 & \dots & 0\\
a_1y & d_{0,1} & 0 & \dots & 0\\
0 & 0 & 1 & \dots & 0\\
\vdots & & & &\\
0 & 0 & 0 & \dots & 1  
\end{bmatrix}
\]
and note that
\[
U_1AV_1
=
\begin{bmatrix}
d_{0,1} & 0 & 0 & \dots & 0 & 0\\
0   &  a_0a_1 & -d_{0,1}a_2 &\dots & 0 & 0\\
\vdots & & & & &\\
0 & 0 & 0 & \dots & a_{k-1} & -a_{k}  
\end{bmatrix}.
\]

We continue in the same manner: let $d_{01,2}=(a_0a_1,d_{0,1}a_2)=d_{0,1}(a_0a_1/d_{0,1},a_2)$, then there exist integers $x_{01,2},y_{01,2}$ such that $d_{01,2}=a_0a_1x_{01,2}+d_{0,1}a_2y_{01,2}$. Embedding this information in another $(k+2)\times (k+2)$ identity matrix $V_2$, we have that
\begin{align*}
U_1AV_1V_2
&=
\begin{bmatrix}
d_{0,1} & 0 & 0 & \dots & 0 & 0\\
0   &  a_0a_1 & -d_{0,1}a_2 &\dots & 0 & 0\\
0   &  0 & a_2 &\dots & 0 & 0\\
\vdots & & & & &\\
0 & 0 & 0 & \dots & a_{k-1} & -a_{k}  
\end{bmatrix}
\begin{bmatrix}
1 & 0 & 0 & \dots & 0\\
0 & x_{01,2}  & a_2d_{0,1}/d_{01,2} & \dots & 0\\
0 & -y_{01,2} & a_0a_1/d_{01,2}  & \dots & 0\\
\vdots & & & &\\
0 & 0 & 0 & \dots & 1  
\end{bmatrix}\\
&=
\begin{bmatrix}
d_{0,1} & 0 & 0 & 0 &\dots & 0 & 0\\
0   &  d_{01,2} & 0 & 0 &\dots & 0 & 0\\
0   &  -a_2y_{01,2} & a_0a_1a_2/d_{01,2} & -a_3 &\dots & 0 & 0\\
\vdots & & & & & &\\
0 & 0 & 0 & 0 & \dots & a_{k-1} & -a_{k}  
\end{bmatrix}
\end{align*}

To get zeros below the leading variable, consider the $(k+1)\times (k+1)$ matrix $U_2$ defined by
\[
U_2
:=
\begin{bmatrix}
1  & 0 & 0 & 0 & \dots & 0\\
0 & 1 & 0 & 0 & \dots & 0\\
0 & -a_2y_{01,2} & d_{01,2} & 0 & \dots & 0\\
\vdots & & & & &\\
0 & 0 & 0 & 0 & \dots & 1  
\end{bmatrix}
\]
and note that
\[
U_2U_1AV_1V_2
=
\begin{bmatrix}
d_{0,1} & 0 & 0 & 0 &\dots & 0 & 0\\
0   &  d_{01,2} & 0 & 0 &\dots & 0 & 0\\
0   &  0 & a_0a_1a_2 & -a_3d_{01,2} &\dots & 0 & 0\\
\vdots & & & & & &\\
0 & 0 & 0 & 0 & \dots & a_{k-1} & -a_{k}  
\end{bmatrix}
\]

Continuing by induction, we see that the Smith Normal Form of the matrix $A$ has diagonal entries $d_{01\cdots j,j+1}$, for $j=0,1,\dots,k-1$, which are defined recursively by
\[
d_{01\cdots j,j+1}:=\gcd(a_0a_1\cdots a_j,a_{j+1} d_{01\cdots j-1,j})
\]
with 
\[
d_{0,1}:=\gcd(a_0,a_1).
\]
What is important to note here is that the SNF of $A$ has the maximal rank; in particular, if the SNF satisfies some nice divisibility properties in relation to the matrix $D=UC$, we get infinitely-many solutions which are parametrized by exactly one free variable (because we have full rank, but the matrix $A$ has dimensions $k\times (k+1)$). Furthermore, the solutions will then be given by 
\[
X=V
\begin{bmatrix}
\frac{b_{11}}{d_1}\\
\frac{b_{22}}{d_2}\\
\vdots\\
\frac{b_{kk}}{d_k}\\
f_{k+1}\\
\end{bmatrix}
\]
and all that is left for us to do is to compute the last column of the matrix $V$. A quick calculation shows that the last column of $V$ has entry
\[
\frac{a_0a_1\cdots a_{i-1}a_{i+1}\cdots a_k}{d_{01\cdots k-1,k}}
\]
in row $i$, for $i=0,1,\dots,k$ and that this is equal to
\[
\frac{\lcm(a_0,a_1,\dots,a_k)}{a_i},
\]
as claimed.

\end{document}